\documentclass[11pt]{amsart}
\scrollmode
\usepackage{amsmath, amsthm,amssymb,amsfonts,mathtools}
\usepackage{verbatim,geometry,graphicx,multirow,array}
\usepackage{algpseudocode}
\usepackage{enumerate}
\usepackage{xcolor}
\usepackage{xy,psfrag}
\usepackage{datetime}
\usepackage{cleveref}

\newcommand{\set}[1]{\left\{#1\right\}}
\newcommand{\Real}{\operatorname{Re}}
\newcommand{\Img}{\operatorname{Im}}

\newcommand{\norm}[1]{\left\Vert #1 \right\Vert}
\newcommand{\lam}{\lambda}

\newcommand{\spn}{\operatorname{span}}

\newcommand{\B}{{\mathbb B}}
\newcommand{\R}{{\mathbb R}}
\newcommand{\C}{{\mathbb C}}
\newcommand{\F}{{\mathbb F}}
\newcommand{\Sph}{{\mathbb S}}

\newcommand{\Cnxn}{{\C^{n\times n}}}
\newcommand{\Ctxt}{{\C^{2\times 2}}}
\newcommand{\bmat}[1]{ \begin{bmatrix}#1\end{bmatrix}}
\newcommand{\smat}[1]{ \left[\begin{smallmatrix} #1 \end{smallmatrix}\right]}
\newcommand{\cont}{{\mathcal C}}
\newcommand{\diag}{\operatorname{diag}}
\newcommand{\rank}{\operatorname{rank}}

\newcommand{\discr}{\operatorname{discr}}
\newcommand{\Rn}{{\R^n}}
\newcommand{\Rp}{{\R^p}}
\newcommand{\Cn}{{\C^n}}
\newcommand{\eps}{\varepsilon}
\newcommand{\ddt}{\frac{d}{dt}}

\newcommand{\Cnn}{{\C^{n \times n}}}
\def \Cmn {{\C^{m\times n}}}
\newcommand{\abs}[1]{\left| #1 \right|}

\renewcommand{\ss}{\scriptstyle}
\newcommand{\Cov}{\left\{ X_s \right\}_{s\in[0,1]}}

\def\sddots{\mathinner{\raise3pt\vbox{\hbox{$\ss .$}}
		\raise1.5pt\hbox{$\ss .$}\hbox{$\ss .$}}}
\let\hat\widehat
\theoremstyle{plain}
\newtheorem{thm}{Theorem}[section]
\newtheorem{lem}[thm]{Lemma}
\newtheorem{cor}[thm]{Corollary}

\newtheorem{rem}[thm]{Remark}
\newtheorem{rems}[thm]{Remarks}
\newtheorem{defn}[thm]{Definition}
\newtheorem{exm}[thm]{Example}

\begin{document}

\title{SVD, joint-MVD, Berry phase, and generic loss of rank\\
		for a matrix valued function of 2 parameters}
	\thanks{We gratefully acknowledge the hospitality of the School of Mathematics
		of Georgia Tech and  the support provided by the University of Bari and by
		INdAM-GNCS}

\author[Dieci]{Luca Dieci}
\address{School of Mathematics, Georgia Institute of Technology,
Atlanta, GA 30332 U.S.A.}
\email{dieci@math.gatech.edu}
\author[Pugliese]{Alessandro Pugliese}
\address{Dept. of Mathematics, University of Bari ``Aldo Moro'', Via Orabona 4, 70125 Bari, Italy}
\email{alessandro.pugliese@uniba.it}
\subjclass{15A18, 15A23, 15A99}

\null\hfill {\fontsize{10}{10pt} \selectfont 
	Version of \today }

\keywords{Matrices depending on parameters, loss of rank, minimum variation 
decomposition}

\begin{abstract}
In this work we consider generic losses of rank for complex valued matrix
functions depending on two parameters.  We give theoretical results that
characterize parameter regions where these losses of rank occur.
Our main results consist in showing how following
an appropriate smooth SVD along a closed loop it is possible to
monitor the Berry phases accrued by the singular vectors to decide if --inside the
loop-- there are parameter values where a loss of rank takes place.
It will be needed to use a new construction of a smooth SVD, which we call
the ``joint-MVD'' (minimum variation decomposition).  
\end{abstract}

\maketitle

\pagestyle{myheadings}
\thispagestyle{plain}
\markboth{L.~Dieci, A.~Pugliese.}{Generic loss of rank}

{\bf Notation.} We indicate with $\Omega$ an open and simply connected subset of 
$\R^2$ or $\R^3$. For points in $\Omega$,
the symbol $\xi$ will indicate $\xi=(x,y)$ if $\Omega\subset\R^2$ or $\xi=(x,y,z)$ if 
$\Omega\subset\R^3$. If $A$ is a complex matrix valued function having $k\ge1$
continuous derivatives on $\Omega$, we write $A\in\cont^k(\Omega,\Cnxn)$ and call 
$A$ \emph{smooth}, and (to avoid trivialities), $n\ge 2$ throughout.
Unless stated otherwise, we will label singular values of a matrix 
$A\in\Cmn$, $m\ge n$, 
in decreasing order $\sigma_1(A)\ge\ldots\ge \sigma_n(A)\ge 0$, and do 
the same for the eigenvalues of a Hermitian matrix $A\in \Cnxn$:
$\lam_1(A)\ge\ldots\ge\lambda_n(A)$. Vectors $v\in\R^n$ are always column vectors. 
The notation $A^*$ indicates the conjugate transpose of $A$.
\section{Introduction and background}\label{sec:intro}
Loss of rank of a matrix $A\in \Cmn$, $m\ge n$, is an issue of paramount importance in
linear algebra, underpinning the concerns of unique solution of a linear
system and the equivalent problems of detecting linear independence of a set
of vectors and of redundancies in data sets.
From a numerical analysis perspective (hence, in finite precision), and ignoring the
concerns of computational expense, it is widely
accepted that the SVD (singular value decomposition) of $A$ is the most reliable and
flexible tool to detect the rank of a matrix, 
for square as well as rectangular matrices\footnote{E.g., for square $A$ (i.e., $m=n$), 
criteria based on the smallest singular value of $A$ are
much more robust than going through an LU-factorization of $A$ and monitoring the
determinant; this is even more true when $A$ is rectangular and one should not
form $A^*A$.}.  Our goal in this work is to understand how the SVD can, and should,
be used to detect losses of rank for matrix valued functions $A$ smoothly
depending on parameters.  

Of course, to be interesting and doable, parameter values
where a loss of rank occurs should be isolated in parameter space, and moreover
we will want to consider problems depending on the minimal possible number of
parameters for the phenomenon to occur.  A very simple counting of the number
of degrees of freedom gives Table \ref{Cod_table} for real and complex valued $A$
of size $(m,n)$, $m\ge n$, in order to have $\rank(A)=n-1$.
\begin{table}
\begin{tabular}{|c|c|}
\hline
$\F$ 
& codimension $\rank(A)= n-1$ \\
\hline
$\R$   &  $m-n+1$ \\
$\C$  &  $2(m-n)+2$ \\
\hline
\end{tabular}
\caption{Values of the codimension for $A\in \F^{m\times n}$, $m\ge n$, to
have $\rank(A)=n-1$ in the two cases of $\F=\R$ and $\F=\C$.}
\label{Cod_table}
\end{table}
The real case tells us that we should expect a loss of rank already when $m=n$ and
$A$ depends on one real parameter (after all, this is detected by the scalar relation
$\det A=0$).  This case is fairly well understood and already
adequately discussed in \cite{BBMN, DE99}, and see also \cite{DGP-SVD} for numerical
methods able to detect and bypass the losses of rank of a smooth function $A$.
The complex case is what we will consider in this work when $m=n$, which has the
minimal possible codimension of $2$ for a single loss of rank.
Again, this setup is easily understood since $\det(A)=0$ are now two conditions, for 
the real and imaginary parts of the determinant.
For the above reasons, we will consider losses of rank for
$A\in\cont^1(\Omega,\Cnxn)$, where $\Omega \subset \R^2$ is open and simply
connected.  For a Hermitian function $A$, however, we will take
$\Omega\subset \R^3$, see Section \ref{sec2}.
%

Finally, we will also require points of loss of rank to be 
\emph{generic}, a  property which we define below.  First, recall that a value 
$v\in\Rn$ is a \emph{regular zero} for a 
smooth map $F:\R^n\rightarrow\R^n$ if $F(v)=0$ and the derivative of $F$ at $v$
is invertible.
\begin{defn}\label{def:gen_loss_of_rank} A point $\xi_0\in\Omega$ is a 
\emph{generic point of loss of rank} for $A\in\cont^1(\Omega,\Cnxn)$ if it is a regular
zero for the map
\begin{equation*}
\xi\in\Omega\mapsto
\bmat{\Real(\det(A(\xi))) \\ \Img(\det(A(\xi)))}\in\R^2.
\end{equation*}
\end{defn}

Our main contribution in this paper will be to device a topological test for the
detection of generic points where a matrix loses rank, a test which also lends
itself to a nice algorithmic criterion to detect regions where $A$ loses rank.
Our test will be based on an appropriate generalization of the concept of 
Berry phase\footnote{ordinarily associated to the eigenvectors of a
Hermitian function, see \cite{Berry} and below}, by looking at the phase accrued by the singular vectors
of a general function $A$.  This will necessitate to find, smoothly, a certain
SVD along a closed path, following what we will call the {\sl joint-MVD} along
the path.  The definition of the joint-MVD is new and to understand it
properly we will revisit the Berry phase of a Hermitian eigenproblem and  
adopt a novel characterization for generic coalescence of 
eigenvalues of the Hermitian eigenproblem.

\begin{rem}
At this stage, we point out 
that working with the Hermitian eigenproblem for $A^*A$ will not help in
finding a useful way
to characterize parameter values where a coalescing occurs, regardless of the
numerical concerns caused by forming the product.  This is already evident in the 
1-parameter case for a real square $A$, whereby through a generic coalescing the
function $\det(A)$ will change sign, but $\det(A^TA)$ will not.
Indeed, the net effect of a reformulation like the one above is to turn a 
generic problem into a non-generic one.   
%
\end{rem}
A plan of the paper is as follows.  Section \ref{sec2} is both a review and a revisitation
of the Hermitian eigenproblem and of generic coalescing of eigenvalues and of its
relation to the Berry phase accumulated by an eigenvector associated to coalescing
eigenvalues.
Section \ref{sec3} is devoted to the joint-MVD and losses of rank, and here we give
our main result, Theorem \ref{thm:main_result} and discuss some
of its consequences.

\section{Hermitian problems: \\
	 generic coalescing of eigenvalues and the Berry phase}\label{sec2}

%
%
%
In this section we consider $A\in \cont^k(\Omega, \Cnxn)$, $k\ge 0$;
typically $A$ will be Hermitian and $\Omega\subset \R^3$.

In general, 
it is well known that a continuous matrix function $A$ taking values in $\Cnxn$
has continuous eigenvalues.
Likewise, it is also well known that if $A$ is a smooth ($k\ge 1$)
Hermitian matrix valued function on $\Omega$, and $\lam_1(A)\ge \ldots\ge
\lam_n(A)$ are
its eigenvalues, then $A$ admits a Schur decomposition $A=U\Lambda U^*$ with 
smooth factors as long as its eigenvalues are distinct everywhere on $\Omega$;
further, in this case, the unitary factor $U$ is 
unique up to post-multiplication by a diagonal unitary matrix $\Phi=\diag(e^{i\alpha_1},\ldots,e^{i\alpha_n})$, 
where each $\alpha_j$ is a smooth real valued function defined on $\Omega$.
A similar result holds for a block decomposition of $A$.  That is, if $A$ has two
(or more) blocks of eigenvalues $\Lambda_1$ and $\Lambda_2$, of size $n_1, n_2$,
that stay disjoint everywhere, then there is a smooth
factorization $A=U\smat{B & 0 \\ 0 & C}U^*$ where $B$ and $C$ are Hermitian
$n_1\times n_1$ and $n_2\times n_2$, respectively, and have eigenvalues
given by those in $\Lambda_1$ and $\Lambda_2$ (e.g., see 
\cite{Ging:GlobBlockBid, HsiehSib:GlobAna}).
%

As it is well understood, the situation is very different when $A$ has a pair of
eigenvalues  
that coalesce and one can end up with no smoothness at all for the eigendecomposition
of $A$ (e.g., see \cite{Kato:Small}).  
This is why it is important to be able to locate parameter values where eigenvalues
coalesce, and it is further mandatory
to focus only on those parameter values where coalescing of eigenvalues occur
in a generic way.  As we noted in \cite{DP_LAA}, generic coalescing of eigenvalues
of a Hermitian function is a co-dimension 3 phenomenon, which we characterize
next.

\begin{defn}[Generic coalescence]\label{def:generic_coal_eig} Let $A\in\cont^1(\Omega,\Cnn)$ be Hermitian, 
and $\Omega$ be an open subset of $\R^3$. 
Let $\lam_j(\xi)=\lam_j(A(\xi))$ be the continuous eigenvalues of $A$, with
$\lam_1(\xi)\ge \cdots \ge \lam_n(\xi)$. Suppose 
\begin{equation*}
\lam_j(\xi)=\lam_{j+1}(\xi) \text{ if and only if } j=h \text{ and } \xi=\xi_0\in\Omega.
\end{equation*}
Then, $\xi_0$ is said to be a \emph{generic point of coalescence} for the eigenvalues
of $A$ according to the following. 
\begin{enumerate}
\item[i)] If $n=2$, write $A(\xi)=\bmat{a(\xi) & b(\xi)+i c(\xi) \\ b(\xi)-i c(\xi) & d(\xi)}$, 
where $a,b,c,d$ are real valued functions, and let 
\begin{equation}\label{eq:F_defn}
F(\xi)=\bmat{a(\xi)-d(\xi) \\ b(\xi) \\ c(\xi)}.
\end{equation}
Then $\xi_0$ is a generic point of coalescence for the eigenvalues of $A$ if it is a 
regular zero for $F(\xi)$.
\item[ii)] If $n>2$, let $R\subset\Omega$ be a pluri-rectangular domain containing $\xi_0$ in its 
interior, and let 
 \begin{equation}\label{eq:block_schur}
 A(\xi)=U(\xi)\bmat{P(\xi) & 0 \\ 0 & \Lambda(\xi)}U^*(\xi)
\end{equation}
be a $\cont^k$ block Schur decomposition of $A(\xi)$ on $R$, where
\begin{equation*}
\begin{split}
& \Lambda(\xi)=\diag(\lambda_1(\xi),\ldots,\lambda_{h-1}(\xi),
\lambda_{h+2}(\xi),\ldots,\lambda_n(\xi))\in\R^{(n-2)\times (n-2)}, \\
& P(\xi)\in\Ctxt \text{ has eigenvalues } \{\lambda_h(\xi),\lambda_{h+1}(\xi)\} 
\text{ for all } \xi\in R.
\end{split} 
\end{equation*}
Then, $\xi_0$ is a generic point of coalescence for the eigenvalues of $A$ if it is a 
generic point of coalescence for the eigenvalues of $P$ according to point i) above.
\end{enumerate}
\end{defn}

Next, we give an alternative condition to characterize genericity of coalescing 
eigenvalues of an Hermitian function, in a way that will be conducive to 
characterize generic losses of rank in Section \ref{sec3}, see Theorem
\ref{thm:altern_gen_loss_of_rank}.
The stepping stone is the next result, characterizing a regular zero of a
$\cont^1$ function.

\begin{lem}\label{lem:hospital} Let $F:\R^p\rightarrow\Rn$, $p\ge 1$, 
be $\cont^1$. Then $x\in\R^p$ is a regular zero for $F$ if and only if $F(x)=0$ and
\begin{equation*}
\lim_{t\rightarrow 0}\frac{\norm{F(x+tv)}_2^2}{t^2}>0 \text{ for any non-zero } v\in\Rp.
\end{equation*}
\end{lem}
 \begin{proof} Note that 
$$
\ddt \norm{F(x+vt)}_2^2=2F(x+vt)^T DF(x+vt) v,
$$
where $DF$ is the derivative of $F$. Since $F(x)=0$, we can write
\begin{equation*}
\begin{split}
\lim_{t\rightarrow 0} \frac{2F(x+vt)^T DF(x+vt) v}{2t} & =
\lim_{t\rightarrow 0} \frac{F(x+vt)^T-F(x)^T}{t} DF(x+tv) v\\
=\ddt F(x+vt)^TDF(x)v & = v^TDF(x)^TDF(x)v  =\norm{DF(x)v}_2^2\ .
\end{split}
\end{equation*}
Therefore, upon using the L'Hospital's rule we get:
$$
\lim_{t\rightarrow 0}\frac{\norm{F(x+tv)}_2^2}{t^2}=
\lim_{t\rightarrow 0}\frac{\ddt \norm{F(x+vt)}_2^2}{2t}=\norm{DF(x)v}_2^2,
$$
from which the statement of the Lemma follows. 
\end{proof}

In Theorem \ref{thm:altern_gen_coal}, we will use Lemma \ref{lem:hospital} applied
to the discriminant of a Hermitian function, by relating genericity of coalescence of
eigenvalues to local properties of the discriminant. 

\begin{defn}[Discriminant] Let $A\in\Cnn$ have eigenvalues $\lam_1,\ldots,\lam_n$. Then the \emph{discriminant} of $A$ is defined as
\begin{equation*}
 \discr(A)=\prod_{\ell<j}(\lam_j-\lam_\ell)^2.
\end{equation*}
\end{defn}
\begin{rem}
If $A$ is Hermitian, then $\discr(A)$ is real valued and non-negative.  Further, 
$\discr(A)=0$ if and only if $A$ has a pair of repeated eigenvalues.  Also,
$\discr(A)$ is a homogeneous polynomial in the entries of $A$.
Therefore, if $A$ is a smooth function of $\xi\in\R^p$, then so is $\discr(A)$.
\end{rem}

\begin{thm}\label{thm:altern_gen_coal} Let $A\in\cont^1(\Omega,\Cnn)$ be Hermitian, 
and $\Omega$ be an open subset of $\R^3$. 
Let $\lam_j(\xi)=\lam_j(A(\xi))$ be the continuous eigenvalues of $A$, with
$\lam_1(\xi)\ge \cdots \ge \lam_n(\xi)$. Suppose 
\begin{equation*}
\lam_j(\xi)=\lam_{j+1}(\xi) \text{ if and only if } j=h \text{ and } \xi=\xi_0\in\Omega.
\end{equation*}
Then, $\xi_0$ is a generic point of coalescence for the eigenvalues of $A$ if and only if
\begin{equation*}
 \lim_{t\rightarrow 0}\frac{\discr(A(\xi_0+tv))}{t^2}>0, \text{ for any non-zero } v\in\R^3.
\end{equation*}
\end{thm}
\begin{proof}
Let 
\begin{equation*}
	A(\xi)=U(\xi)\bmat{P(\xi) & 0 \\ 0 & \Lambda(\xi)}U^*(\xi) 
\end{equation*}
for all $\xi$ inside a pluri-rectangle $R$ whose interior contains $\xi_0$, as in 
\eqref{eq:block_schur}. Then, we can write
\begin{equation*}
	\discr(A(\xi))= \discr(P(\xi)) \prod_{
		\substack{
			j<\ell \\
			(j,\ell)\ne(h,h+1)}}(\lam_j(\xi)-\lam_\ell(\xi))^2=\discr(P(\xi))g(\xi),
\end{equation*}
where $g$ is defined by the above equation. Note that $g$ is a smooth and 
strictly positive function of $\xi$. Then, 
we can write
\begin{equation*}
	\lim_{t\rightarrow 0}\frac{\discr(A(\xi_0+tv))}{t^2}=
	g(\xi_0)\lim_{t\rightarrow 0}\frac{\discr(P(\xi_0+tv))}{t^2}, \text{ for any non-zero } v\in\R^3.
\end{equation*}
The result follows from Lemma \ref{lem:hospital} with $F$ as in \eqref{eq:F_defn} since, by letting $P=\bmat{a & b+ic \\ b-ic & d}$, we have
$\discr P = (a-d)^2+b^2+c^2$.
\end{proof}

In Section \ref{sec3}, we will need 
the following elementary result, which will be useful to relate
a generic coalescing of eigenvalues to a generic loss of rank.

\begin{lem}\label{lem:M_eigenvalues_and_discr} Let $A\in\Cnxn$ and $\eps\in\R$, and consider the Hermitian matrix function
\begin{equation}\label{eq:M_def}
M=\bmat{\eps I & A \\ A^* & -\eps I}.
\end{equation}
Then $M$ has eigenvalues $\pm\sqrt{\sigma_j^2+\eps^2}$, and
\begin{equation}\label{eq:discr_M}
\discr(M) = 4^n\prod_{j<\ell}(\sigma_j^2-\sigma_\ell^2)^4\prod_{j}(\sigma_j^2+\eps^2).
\end{equation}
\end{lem}
\begin{proof} A direct computation gives
\begin{equation*}
 \det(tI - M)=\det (t^2 I-(A^*A+\eps^2 I)),
\end{equation*}
from which the two statements follow. 
\end{proof}


We conclude this section with some known results (mostly from  \cite{DE99, DP_LAA})
that allow us to lay the groundwork for detecting losses of rank in Section \ref{sec3}.

\subsection{Hermitian 1 parameter, Berry phase, and covering of a sphere}
First, consider the case of Hermitian $A$ smoothly depending on one 
parameter: $A\in\cont^k(\R,\Cnn)$. In this case, there is a standard way of 
resolving the degree of non-uniqueness of its Schur decomposition, which in the end 
leads to the concept of \emph{Berry phase}. We summarize this as follows, see \cite{DP_LAA}.

\begin{thm}\label{thm:MVD}
Let $A\in\cont^k([0,1],\Cnn)$, $k\ge 1$, be Hermitian with distinct eigenvalues 
$\lam_1(t)>\ldots>\lam_n(t)$ for all $t\in[0,1]$. Then, given a Schur decomposition of 
$A$ at $t=0$, $A(0)=U_0\Lambda_0U_0^*$, there exists a uniquely defined so called 
\emph{Minimum Variation Decomposition (MVD)} $A(t)=U(t)\Lambda(t)U^*(t)$, 
$t\in[0,1]$, satisfying $U(0)=U_0$, $\Lambda(0)=\Lambda_0$, where $U$ minimizes 
the total variation
\begin{equation}\label{eq:MVD_integral}
 \operatorname{Vrn}(U)=\int_0^1\norm{\dot U(t)}_{\mathrm{F}}dt
\end{equation}
among all possible smooth unitary Schur factors of $A$ over the interval $[0,1]$.

In addition, suppose that $A$ is $1$-periodic and of minimal period $1$.  Then,
we have:
\begin{itemize}
\item[i)]  $U$ satisfies
\begin{equation}\label{eq:Berry_defn}
 U(0)^*U(1)=\diag(e^{i\alpha_1},\ldots,e^{i\alpha_n}),
\end{equation}
where each $\alpha_j\in(-\pi,\pi]$, $j=1,\ldots,n$, is the so called \emph{Berry phase} 
associated to $\lambda_j$;
\item[ii)] if $Q\in\cont^k([0,1],\Cnn)$ is a $1$-periodic unitary Schur factor for $A$ over 
$[0,1]$, partitioned by columns $Q=[q_1, \cdots, q_n]$, then
\begin{equation}\label{eq:Berry_formula}
 \alpha_j=i\int_0^1 q_j^*(t)\dot{q}_j(t) dt \mod 2\pi,\ \text{for all }j=1,\ldots,n.
\end{equation}
\end{itemize}
\end{thm}

For our purposes, the relevance of the Berry phase is because of its relation
to detection of coalescing eigenvalues in a region of $\R^3$, homotopic to
a sphere, as we recall next.

Let $\Sph_r=\set{\xi\in\R^3: \norm{\xi}_2=r}$ be the sphere of radius $r>0$ centered 
at the origin in $\R^3$, and consider
for $\Sph_r$ the following parametrization:
\begin{equation}\label{eq:spherical_coord}
	\left\{ \begin{array}{l} x(s,t)=r\sin(\pi s)\cos(2\pi t) \\ y(s,t)=r\sin(\pi s)\sin(2\pi t) \\
		z(s,t)=r\cos(\pi s)  \end{array},\right.
\end{equation}
with $(s,t)\in[0, 1]\times[0, 1]$. The sphere $\Sph_r$ can be thought of
as covered by the family of loops $\Cov$,
$$X_s(\cdot)=(x(s,\cdot), y(s,\cdot), z(s,\cdot)),$$
as $s$ increases from $s=0$ to $s=1$. 

Let $A:\xi\in\R^3\mapsto A(\xi)\in\Cnn$ be a $\cont^k$ Hermitian matrix valued 
function, and suppose that all eigenvalues of $A$ are distinct on some sphere 
$\Sph_r$, $r>0$.  Then, the restriction of $A$ to each loop in $\Cov$ is a $1$-periodic 
function and therefore, according to Theorem \ref{thm:MVD}, each eigenvector of $A$ 
continued along $X_s$ accrues a Berry phase $\alpha_j(s)$, $j=1,\ldots,n$. All 
$\alpha_j(s)$'s can be defined to be continuous functions of $s$, again see \cite{DP_LAA}. Moreover, since 
$X_0$ and $X_1$ are just points, the corresponding MVD (see Theorem 
\ref{thm:MVD}) of $A$ must have constant factors, and therefore
\begin{equation*}
	\begin{array}{c}
		\alpha_j(0)=0\mod 2\pi, \\
		\alpha_j(1)=0\mod 2\pi.
	\end{array}
\end{equation*}

Let $\B_r$ be the solid ball $\B_r=\set{\xi\in\R^3: \norm{\xi}_2\le r}$,
so that $\Sph_r$ is the boundary of $\B_r$.


\begin{thm}{\rm (Adapted from
\cite[Theorems 4.6, 4.8 and 4.10]{DP_LAA})}
	\label{thm:generic_coal_nxn}
	Let $A\in\cont^1(\B_r,\Cnxn)$ be Hermitian, and let 
	$\lambda_1(\xi), \lambda_2(\xi), \ldots, \lambda_n(\xi)$ be its continuous 
	eigenvalues, not necessarily labelled in descending order, and let
	$\alpha_j(s)$, $s\in[0,1]$, be the continuous Berry phase functions associated to 
	$\lam_j$ over $\Sph_r$, for all $j=1,\ldots,n$.
	\begin{itemize}
		\item[(i)] If $\lambda_1(\xi), \lambda_2(\xi), \ldots, \lambda_n(\xi)$ are 
		distinct for all $\xi\in\Sph_r$, then
		\begin{equation*}
			\sum_{j=1}^n\alpha_j(s)=\sum_{j=1}^n\alpha_j(0), \text{ for all } s\in[0,1],
		\end{equation*}
		\item[(ii)]  If $\lambda_1(\xi), \lambda_2(\xi), \ldots, \lambda_n(\xi)$ are distinct
		for all $\xi\in\B_r$, then\begin{equation*}
	\alpha_j(1)=\alpha_j(0)  \text{ for all } j=1,\ldots n.
\end{equation*}
		\item[(iii)]  Finally, suppose that $\lambda_j(\xi)=\lambda_{k}(\xi)$ if and only if 
		$(j,k)=(h_1,h_2)$ and $\xi=0$, and 
		that $\xi=0$ is a generic point of coalescence for the eigenvalues of $A$. Then:
		\begin{equation*}
			\left\{
			\begin{array}{ll}
				\alpha_j(1)=\alpha_j(0), & \text{ for all } j\ne h_1,h_2\\
				\alpha_{h_1}(1)=\alpha_{h_1}(0)\pm2\pi, & \\
				\alpha_{h_2}(1)=\alpha_{h_2}(0)\mp2\pi. & 
			\end{array}
			\right.
		\end{equation*}
	\end{itemize}
\end{thm}

\section{SVD, joint-MVD, and generic losses of rank}\label{sec3}

We are ready to characterize generic losses of rank
for a smooth general matrix function of two
parameters, $A=A(x,y)$ (see Definition \ref{def:gen_loss_of_rank}).
In particular, in  Theorem \ref{thm:altern_gen_loss_of_rank}
we will relate a generic loss of rank to the
local behavior of the smallest singular value of $A$.

\begin{thm}\label{thm:altern_gen_loss_of_rank} 
	A point $\xi_0\in\Omega$ is a \emph{generic point of loss of rank} for $A\in\cont^1(\Omega,\Cnxn)$ if and only if
	\begin{equation*}
		\lim_{t\rightarrow 0^+}\frac{\sigma_n(A(\xi_0+tv))}{t}>0 \text{ for any non-zero } v\in\R^2.
	\end{equation*}
\end{thm}
\begin{proof}
	Consider the following $2n\times 2n$ Hermitian matrix
	\begin{equation*}
		B(\xi)=\bmat{0 & A(\xi) \\ A^*(\xi) & 0}.
	\end{equation*}
	Its eigenvalues are $\pm\sigma_1(\xi),\ldots,\pm\sigma_n(\xi)$, and therefore we can write
	\begin{equation*}
		\det(B(\xi))=(-1)^n\abs{\det(A(\xi))}^2=(-1)^n\prod_j\sigma_j(\xi)^2.
	\end{equation*}
	On the other hand, we have
	\begin{equation*}
		\det(B(\xi))=(-1)^n\left(\Real(\det(A(\xi)))^2+\Img(\det(A(\xi)))^2\right).
	\end{equation*}
	Let
	\begin{equation*}
		F(\xi)=\bmat{\Real(\det(A(\xi))) \\ \Img(\det(A(\xi)))}\in\R^2.
	\end{equation*}
	Then, by virtue of Definition \ref{def:gen_loss_of_rank} and Lemma \ref{lem:hospital}, $\xi_0$ is a generic point of loss of rank if and only if
	\begin{equation*}
		\lim_{t\rightarrow 0}\frac{\norm{F(\xi_0+tv)}_2^2}{t^2}>0 \text{ for any non-zero } v\in\R^2,
	\end{equation*}
	and this is equivalent to 
	\begin{equation*}
		\lim_{t\rightarrow 0}\frac{(-1)^n\det(B(\xi_0+tv))}{t^2}>0 \text{ for any non-zero } v\in\R^2.
	\end{equation*}
	Since the ordered singular values are continuous, and we have $\sigma_1\ge\ldots\ge\sigma_{n-1}>0$ in a 
	neighborhood of $\xi_0$, we can write
	\begin{equation*}
		\lim_{t\rightarrow 0}\frac{\sigma_n^2(A(\xi_0+tv))}{t^2}>0 \text{ for any non-zero } v\in\R^2.
	\end{equation*}
    Then, the sought statement follows from taking the square root of the limit. 
\end{proof}

Next, in Theorem 
\ref{thm:main_genericity_equiv_theorem},  we relate a generic loss of
rank to the coalescing
of the eigenvalues of a Hermitian function of 3 parameters.

\begin{thm}\label{thm:main_genericity_equiv_theorem}
	Let $A=A(\xi)\in\cont^1(\Omega,\Cnxn)$, $\Omega\subset\R^2$,
	have distinct singular values for all $\xi\in\Omega$. Consider
	\begin{equation*}
		M(\eta)=\bmat{\eps I & A(\xi) \\ A^*(\xi) & -\eps I},
	\end{equation*}
	where $\eta=(\xi,\eps)\in\Omega\times\R$. Let $\xi_0$ be the only point in $\Omega$ 
	where $A$ loses rank. Then, $\xi_0\in\Omega$ is a generic point of loss of rank for 
	$A$ if and only if $\eta_0=(\xi_0,0)$ is a generic point of coalescence for the 
	eigenvalues of $M$. 
\end{thm}
\begin{proof} We will show that Theorems \ref{thm:altern_gen_loss_of_rank} and 
	\ref{thm:altern_gen_coal} are equivalent through Lemma 
	\ref{lem:M_eigenvalues_and_discr}. First, note that $\eta_0$ is the only point in 
	$\Omega\times\R$ where two eigenvalues of $M$ coalesce. Obviously, the 
	coalescing pair is $\pm\sqrt{\sigma_n^2+\eps^2}$.
	
	Because of Theorem \ref{thm:altern_gen_loss_of_rank}, $\xi_0$ is a generic point 
	of loss of rank for $A$ if and only if
	\begin{equation}\label{eq:main_genericity_equiv_theorem_eq}
		\lim_{t\rightarrow 0}\frac{\sigma_n^2(A(\xi_0+tv))}{t^2}>0 \text{ for any non-zero 
		} v\in\R^2
	\end{equation}
	and this is equivalent to 
	\begin{equation*}
		\lim_{t\rightarrow 0}\frac{\sigma_n^2(A(\xi_0+tv))+\gamma^2t^2}{t^2}>0 \text{ for any non-zero }(v,\gamma)\in\R^2\times\R.
	\end{equation*}
	Now, equation \eqref{eq:discr_M} can be rewritten as
	\begin{equation*}
		\discr(M) = \left(4^n\prod_{j<\ell}(\sigma_j^2-\sigma_\ell^2)^4\prod_{j=1}^{n-1}(\sigma_j^2+\eps^2)\right)
		(\sigma_n^2+\eps^2),
	\end{equation*}
	and the first of the two factors of this product is strictly positive in
	$\Omega\times\R$. Then, through Lemma \ref{lem:M_eigenvalues_and_discr}, 
	Equation \eqref{eq:main_genericity_equiv_theorem_eq} is equivalent to
	\begin{equation}
		\lim_{t\rightarrow 0}\frac{\discr(M(\eta_0+tv))}{t^2}>0, \text{ for any non-zero } v\in\R^3,
	\end{equation}
	which, by Theorem \ref{thm:altern_gen_coal}, expresses the fact that $\eta_0$ is a 
	generic point of coalescence for the eigenvalues of $M$. 
\end{proof}

%

We will leverage the relation between losses of rank of $A$ and coalescing of 
eigenvalues of $M$ of Theorem \ref{thm:main_genericity_equiv_theorem}, but of
course without forming $M$ but working directly with an appropriate SVD of $A$.
The stepping stone will be Theorem \ref{thm:eigendec_for_M}, for whose proof the
next Lemma will be handy.
\begin{lem}\label{lem:invar_sub}
	Let $A\in\Cnn$ be diagonalizable. Let $\lam\in\C$ be an eigenvalue of $A$ of 
	multiplicity 1 such that $\lam^2$ is an eigenvalue of $A^2$ of multiplicity 2. Let 
	$u,v\in\Cn$ span the eigenspace of $A^2$ associated to the eigenvalue $\lam^2$. 
	Then, the eigenspace of $A$ associated to $\lam$ is spanned by a linear 
	combination of $u$ and $v$. 
\end{lem}
\begin{proof} If $A$ had an eigenvector not in $\spn\{u,v\}$, then $A^2$ would have a 
3-dimensional invariant subspace associated to $\lam^2$, contradicting the 
hypothesis on the multiplicity of $\lam^2$. 
\end{proof}

\begin{thm}\label{thm:eigendec_for_M} Let $A\in\Cnn$ and $A=U\Sigma V^*$ be a 
	SVD of $A$, where $\Sigma=\diag(\sigma_1,\ldots, \sigma_n)$. Suppose that all 
	singular values of $A$ are distinct and non-zero, and let $\eps\in\R$. 
	Let $M$ be given by \eqref{eq:M_def}.  Then, $M$ admits the following
	eigendecomposition
	\begin{equation*}
		W^* M W = \bmat{S & 0 \\0 & -S},
	\end{equation*}
	where	$S=\diag\left(\sqrt{\sigma_1^2+\eps^2},\ldots,\sqrt{\sigma_n^2+\eps^2}\right)$,
	\begin{equation*}
		W=\bmat{UC & -UD\\VD & VC},
	\end{equation*}
	$C=\diag(c_1,\ldots,c_n)$, $D=\diag(d_1,\ldots,d_n)$, and the diagonal entries of 
	$C$ and $D$ are
	given by
        \begin{equation}\label{eq:CD_explicit_formulae}
	\begin{cases}
		c_j = \dfrac{1}{\sqrt{2}}\dfrac{\sigma_j}{\sqrt{\sigma_j^2 + \eps^2 - \eps\sqrt{\sigma_j^2 + \eps^2}}} \\
		d_j = \dfrac{1}{\sqrt{2}}\dfrac{\sqrt{\sigma_j^2 + \eps^2} - \eps}{\sqrt{\sigma_j^2 + \eps^2 - \eps\sqrt{\sigma_j^2 + \eps^2}}}
	\end{cases},\ j=1,\ldots,n.
        \end{equation}
\end{thm}
\begin{proof} 
Recall that the eigenvalues of $M$ are $\pm\sqrt{\sigma_j^2+\eps^2}$ and they
are all distinct as long as the singular values of $A$ are all distinct and non-zero.

Let $A=U\Sigma V^*$ be a singular value decomposition of $A$, with 
$\Sigma=\diag(\sigma_1,\ldots, \sigma_n)$ and $U=[u_1,\ldots,u_n]$, 
$V=[v_1,\ldots,v_n]$, be partitioned by columns. Then, we have that
\begin{equation*}
	M^2=\bmat{\eps^2 I+AA^* & 0\\ 0 & \eps^2 I+A^*A},
%
%
	M^2\bmat{u_j \\ 0}=(\sigma_j^2+\eps^2)\bmat{u_j \\ 0},\  M^2\bmat{0 \\ v_j}=(\sigma_j^2+\eps^2)\bmat{0 \\ v_j},
\end{equation*}
for all $j=1,\ldots,n$. If $\sigma_1>\cdots>\sigma_n>0$, it follows from Lemma 
\ref{lem:invar_sub} that, for all $j=1,\ldots,n$, there exist $c_j, d_j\in\C$ not both zero 
such that 
\begin{equation}\label{eq:eigeqn_for_M}
	M\bmat{c_ju_j \\ d_j v_j}=\sqrt{\sigma_j^2+\eps^2}\bmat{c_ju_j \\ d_j v_j}.
\end{equation}
These equations can be rewritten as
\begin{equation}\label{eq:eigeqn_for_M_expl}
	\bmat{(\eps c_j +\sigma_j d_j) u_j \\ (\sigma_j c_j -\eps d_j)v_j } =
	\bmat{c_j\sqrt{\sigma_j^2+\eps^2} u_j\\ d_j\sqrt{\sigma_j^2+\eps^2} v_j},\ j=1,\ldots,n,
\end{equation}
and thus
\begin{equation}\label{eq:sys_eqns_for_CD}
	\begin{cases}
		\left(\eps - \sqrt{\sigma_j^2+\eps^2}\right) c_j +\sigma_j d_j = 0 \\
		\sigma_j c_j - \left(\eps + \sqrt{\sigma_j^2+\eps^2}\right) d_j = 0
	\end{cases},\ j=1,\ldots,n.
\end{equation}
Now, for any $j=1,\ldots,n$, \eqref{eq:sys_eqns_for_CD}
has infinitely many non trivial solutions $(c_j,d_j)$, where $c_j$ and $d_j$ are real 
valued and cannot have opposite sign. Imposing the normalization conditions 
$c_j^2+d_j^2=1$, and settling (without loss of generality) on the positive solutions, we get the expressions for $c_j$ and $d_j$ given in \eqref{eq:CD_explicit_formulae}.
Finally, let $C=\diag(c_1,\ldots,c_n)$, $D=\diag(d_1,\ldots,d_n)$, 
$S=\diag\left(\sqrt{\sigma_1^2+\eps^2},\ldots,\sqrt{\sigma_n^2+\eps^2}\right)$.
Then equations \eqref{eq:eigeqn_for_M} read
$M\bmat{UC \\ VD}=\bmat{UC \\ VD}S$,
and by letting $W=\bmat{UC & -UD\\VD & VC}$,
one easily sees that $W$ is unitary and that
\begin{equation*}
	W^* M W = \bmat{S & 0 \\0 & -S},
\end{equation*}
which is the desired result.
\end{proof}

\begin{rem}\label{rem:CD_parity}
Each $c_j$ and $d_j$ in \eqref{eq:CD_explicit_formulae}
depends smoothly on $\eps$ and on $\sigma_j$, and moreover 
$C(-\eps)=D(\eps)$ and $D(-\eps)=C(\eps)$, while $S(\eps)=S(-\eps)$ 
because of how $S$ is defined. Now, consider:	
	\begin{equation*}
		M(\eps)=
		\bmat{UC(\eps) & -UD(\eps)\\VD(\eps) & VC(\eps)}
		\bmat{S(\eps) & 0\\0 & -S(\eps)}
		\bmat{C(\eps)U^* & D(\eps)V^*\\ -D(\eps)U^* & C(\eps)V^*},
	\end{equation*}	
 and observe that $C(\eps), D(\eps), S(\eps)$ depend smoothly on $\eps$ 
 (since the singular values of $A$ are distinct and non zero). Then one has
	\begin{equation*}
		\begin{split}
			M(-\eps) & =
			\bmat{UC(-\eps) & -UD(-\eps)\\VD(-\eps) & VC(-\eps)}
			\bmat{S(-\eps) & 0\\0 & -S(-\eps)}
			\bmat{C(-\eps)U^* & D(-\eps)V^*\\ -D(-\eps)U^* & C(-\eps)V^*} \\
			& = \bmat{UD(\eps) & -UC(\eps)\\VC(\eps) & VD(\eps)}
			\bmat{S(\eps) & 0\\0 & -S(\eps)}
			\bmat{D(\eps)U^* & C(\eps)V^*\\ -C(\eps)U^* & D(\eps)V^*}.
		\end{split}
	\end{equation*}
\end{rem}

The main result of this paper will show
that a loss of rank is detected by the phases
accumulated by the singular vectors for an appropriate smooth decomposition
of $A(x,y)$ along a closed loop.  To properly define/understand our result,
it is necessary to clarify what ``appropriate'' means, and this requires looking
at how to define/compute a smooth SVD along a closed loop.

\subsection{Smooth SVD: 1 parameter}\label{SVD-1para}
We will follow the approach of \cite{DE99}.  We have a smooth
function $A$, depending on a real parameter $t\in [0,1]$ and with
distinct singular values for all $t$.   Then, 
the SVD factors of $A$ are smooth and satisfy the system of
differential equations given in \eqref{eq:Ck_SVD_DAE}.

%
%
\begin{thm}[Adapted from \cite{DE99}]\label{thm:Ck_SVD} Let 
$A\in\cont^k([0,1],\Cnn)$, $k\ge 1$, have distinct singular values 
$\sigma_1(t)>\ldots>\sigma_n(t)>0$ for all $t\in[0,1]$. Then, given any initial singular 
value decomposition $A(0)=U_0\Sigma_0V_0^*$,
there exists a $\cont^k$ singular value decomposition $A(t)=U(t)\Sigma(t)V^*(t)$, 
$t\in[0,1]$, defined as solution of the following differential-algebraic 
initial value problem:
\begin{equation}\label{eq:Ck_SVD_DAE}
\left\{
\begin{array}{l}
 \dot\Sigma = U^*\dot AV-H\Sigma+\Sigma K ,\\
 \dot U = UH, \\
 \dot V = VK, \\
 U(0) = U_0, \Sigma(0)=\Sigma_0, V(0)=V_0.
\end{array}.
\right.
\end{equation}
The matrix functions $H$ and $K$ are skew-Hermitian on $[0,1]$,
with entries given by
\begin{equation}\label{eq:offdiagHK_expression}
\begin{array}{l}
H_{j\ell}=\dfrac{\sigma_\ell(U^*\dot A V)_{j\ell}+\sigma_j
(U^*\dot A V)_{\ell j}}{\sigma_\ell^2-\sigma_j^2} \\
K_{j\ell}=\dfrac{\sigma_\ell(U^*\dot A V)_{\ell j}+\sigma_j
(U^*\dot A V)_{j\ell}}{\sigma_\ell^2-\sigma_j^2}
\end{array}
\end{equation}
for all $j\ne\ell$. The diagonal entries of $H$ and $K$ are real valued and satisfy
\begin{equation}\label{eq:diagHK_extra_requ}
H_{jj}-K_{jj}=\dfrac{\Img\big((U^*\dot AV)_{jj}\big)}{\sigma_j},
\quad \text{for all}\,\, j=1,\dots,n.
\end{equation}
\end{thm}

\begin{rem}\label{rem:Ck_SVD} Obviously, the requirement 
\eqref{eq:diagHK_extra_requ} does not fully determine the diagonal entries of $H$ and 
$K$ and
we are left with $n$ conditions to impose. 
To uniquely determine a smooth SVD path, one possibility was suggested
in 
\cite{DE99}, simply set $H_{jj}=0$ (or $K_{jj}=0$) for all $j$, and this was shown in 
\cite{DP_LAA} to be equivalent to selecting 
the SVD path that minimizes the total variation of $U$ (or $V$) on $[0,1]$ given in 
\eqref{eq:MVD_integral} and defined originally in \cite{BBMN}; we call these the $U$-MVD or $V$-MVD, respectively.
\end{rem}

None of the options of Remark \ref{rem:Ck_SVD} to select a smooth SVD path
would be useful for our purposes of detecting when $A$ loses rank.  The correct
smooth SVD path for us is identified in the next Theorem.

\begin{thm}\label{thm:jointMVD} Let $A\in\cont^k([0,1],\Cnn)$, $k\ge 1$, have distinct 
	singular values $\sigma_1(t)>\ldots>\sigma_n(t)$ for all $t\in[0,1]$. Then, given any 
	initial singular value decomposition $A(0)=U_0\Sigma_0V_0^*$, there exists a 
	uniquely defined $\cont^k$ singular value decomposition 
	$A(t)=U(t)\Sigma(t)V^*(t)$, $t\in[0,1]$, satisfying $U(0)=U_0, 
	\Lambda(0)=\Lambda_0, V(0)=V_0$ and such that the pair $(U, V)$ minimizes the 
	quantity
\begin{equation}\label{eq:joint_MVD_integral}
\int_0^1\sqrt{\norm{\dot U(t)}^2_{\mathrm{F}}+\norm{\dot V(t)}^2_{\mathrm{F}}}\,dt
\end{equation}
among all possible smooth unitary SVD factors of $A$ over the interval $[0,1]$.
\end{thm}
\begin{proof} 
    Since the Frobenius norm is unitarily invariant, using \eqref{eq:Ck_SVD_DAE} we immediately see that to minimize the quantity in \eqref{eq:joint_MVD_integral} is the same as to minimize 
\begin{equation*}
\int_0^1\sqrt{\norm{U^*(t)\dot U(t)}^2_{\mathrm{F}}+\norm{V^*(t)\dot V(t)}^2_{\mathrm{F}}}\,dt.
\end{equation*}
We now show that all the singular value decompositions satisfying \cref{eq:Ck_SVD_DAE,eq:offdiagHK_expression,eq:diagHK_extra_requ} share the same value for
\begin{equation*}
    \sum_{j\ne\ell}\abs{\big(U^*(t)\dot U(t)\big)_{j\ell}}^2+\abs{\big(V^*(t)\dot V(t)\big)_{j\ell}}^2, t\in[0,1].
\end{equation*}
In fact recall that, being the singular values all distinct, each unitary factor ($U$ or $V$) is 
unique up to post-multiplication by a smooth diagonal unitary matrix function $\Phi(t)=\diag\left(e^{i\phi_1(t)},\ldots,e^{i\phi_n(t)}\right)$. Let $U(t)$ and $Q(t)=U(t)\Phi(t)$ be two matrices of left singular vectors satisfying \cref{eq:Ck_SVD_DAE,eq:offdiagHK_expression,eq:diagHK_extra_requ}. A simple computation shows that 
\begin{equation*}
    \abs{\big(U^*(t)\dot U(t)\big)_{j\ell}}=\abs{\big(Q^*(t)\dot Q(t)\big)_{j\ell}}, \text{ for all } j\ne\ell \text{ and } t\in[0,1].
\end{equation*}
Of course, analogus considerations hold for $V(t)$. Therefore, minimizing \eqref{eq:joint_MVD_integral} is equivalent to minimizing
\begin{equation*}
\int_0^1\sqrt{\norm{\diag\big(U^*(t)\dot U(t)\big)}^2_{\mathrm{F}}+\norm{\diag\big(V^*(t)\dot V(t)\big)}^2_{\mathrm{F}}}\,dt,
\end{equation*}
that is
\begin{multline*}
\int_0^1\sqrt{\sum_{j=1}^{n}\biggl(\abs{H_{jj}(t)}^2+\abs{K_{jj}(t)}^2}\biggr)\,dt\\ =\int_0^1\sqrt{\frac{1}{2}\sum_{j=1}^{n}\biggl(\abs{H_{jj}(t)-K_{jj}(t)}^2+\abs{H_{jj}(t)+K_{jj}(t)}^2}\biggr)\,dt.
\end{multline*}
Since the difference $H_{jj}-K_{jj}$ is prescribed by \eqref{eq:diagHK_extra_requ}, the 
minimizing choice is given by
\begin{equation*}
 H_{jj}(t)+K_{jj}(t)=0, \text{ for all } t\in[0, 1] \text{ and all } j=1,\ldots,n.
\end{equation*}
Using this, along with \cref{eq:Ck_SVD_DAE,eq:offdiagHK_expression,eq:diagHK_extra_requ}, yields (uniquely) the desired unitary factors $U$ and $V$.
\end{proof}

\begin{defn}\label{def:jointMVD} Any smooth singular value decomposition of 
$A\in\cont^k([0,1],\Cnn)$ satisfying \eqref{eq:joint_MVD_integral} will be called a 
\emph{joint minimum variation decomposition}, or simply ``joint-MVD".
\end{defn}

\begin{rems}\label{rem:after_jointMVD}
$\,$%
\begin{itemize}
\item[(i)] If $A$ is Hermitian, then $U$ and $V$ are equal, and unique up to changes 
of sign of their columns. In this case, the joint-MVD of $A$ is effectively the MVD 
of Theorem \ref{thm:MVD}.
\item[(ii)] If $A$ is periodic, then --using the joint-MVD--
each singular vector acquires a phase factor over one period, and the
corresponding left and right singular vectors acquire the same phase.
This can be thought of as a generalization of the Berry phase to non-Hermitian
matrix functions and in fact it is the same value as the
Berry phase accrued by the eigenvectors of $\bmat{0 & A \\ A* & 0}$.
\end{itemize}
\end{rems}


%
%
%
%

\subsection{Main Result}
We can finally formulate and prove the main result of this paper, showing
that a loss of rank inside a closed loop is detected by the phases
accumulated by the singular vectors.

\begin{thm}\label{thm:main_result} Let $A\in\cont^1(\Omega,\Cnxn)$, 
	$\Omega\subset\R^2$, have distinct eigenvalues everywhere in $\Omega$. Suppose 
	that $\xi_0\in\Omega$ is a generic point of loss of rank for $A$, and that $A$ has 
	full rank everywhere else in $\Omega$. Let $\Gamma$ be a circle centered at $\xi_0$ 
	entirely contained in $\Omega$, and let it be parametrized by 
	$\gamma(t)=\xi_0+[r\cos(2\pi t), r\sin(2\pi t)]$, $t\in[0,1]$. Let 
	$A(\gamma(t))=U(t)\Sigma(t)V^*(t)$ be the joint-MVD of $A(\gamma(t))$ over the 
	interval $[0,1]$. Let $\beta_j\in(-\pi,\pi]$, $j=1,\ldots,n$, be defined through the following 
	equation:
\begin{equation*}
 U^*(0)U(1) = V^*(0)V(1) = \diag(e^{i\beta_1},\ldots,e^{i\beta_n}).
\end{equation*}
Then, we have
\begin{equation*}
 \sum_{j=1}^n\beta_j=\pi\mod 2\pi.
\end{equation*}
\end{thm}

\begin{proof} Without loss of generality, we may take $\xi_0=(0,0)$. 
Consider the Hermitian matrix function of three parameters
\begin{equation*}
M(x,y,z)=\bmat{z I & A(x,y) \\ A^*(x,y) & -z I}, (x,y)\in\Omega \text{ and } z\in\R.
\end{equation*}
Because of Theorem \ref{thm:eigendec_for_M}, for all $(x,y,z)$, $M$ has the Schur 
eigendecomposition 
\begin{equation*}
 W(x,y,z)^* M(x,y,z) W(x,y,z) = \bmat{S(x,y,z) & 0 \\0 & -S(x,y,z)},
\end{equation*}
where
\begin{equation*}
\begin{array}{c}
S(x,y,z)=\diag\left(\sqrt{\sigma_1(x,y)^2+z^2},\ldots,\sqrt{\sigma_n^2(x,y)+z^2}\right), 
\\
W(x,y,z)=\bmat{U(x,y)C(x,y,z) & -U(x,y)D(x,y,z)\\V(x,y)D(x,y,z) & V(x,y)C(x,y,z)}.
\end{array}
\end{equation*}
Let us label the eigenvalues $\lambda_1,\ldots,\lambda_{2n}$ of $M$ in the same order 
as they appear along the diagonal of $\bmat{S & 0\\ 0 & -S}$, that is so that 
$\lambda_j=\sqrt{\sigma_j^2+z^2}$, $\lambda_{n+j}=-\sqrt{\sigma_j^2+z^2}$, for 
$j=1,\ldots,n$.

Consider the sphere $\Sph_r$ parametrized by $(s,t)\in[0,1]\times[0,1]$ as in 
\eqref{eq:spherical_coord}. 
It follows from Theorem \ref{thm:main_genericity_equiv_theorem} that $M$ and its 
eigenvalues satisfy the hypotheses of 
Theorem \ref{thm:generic_coal_nxn}-(i,ii) on $\B_r$, with the pair of eigenvalues that 
undergoes coalescence being $(\lambda_n,\lambda_{2n})$. 
For each $j=1,\ldots,2n$, let $\alpha_j(s)$, $s\in[0,1]$, be the continuous Berry phase 
function associated to $\lam_j$ over $\Sph_r$, where we choose $\alpha_j(0)=0$ for 
all $j=1,\ldots,2n-1$, and $\alpha_{2n}(0)=2\pi$. Then, Theorem 
\ref{thm:generic_coal_nxn} gives
\begin{equation}\label{eq:sum_of_alphajs}
\sum_{j=1}^{2n}\alpha_j(s)=2\pi, \text{ for all } s, \quad \text{and}\quad 
\left\{
\begin{array}{ll}
	\alpha_j(1)=0, & \text{ for all } j\ne n, 2n, \\
	\alpha_{n}(1)=2\pi, & \\
	\alpha_{2n}(1)=0. & 
\end{array}
\right.
\end{equation}
Now, let
\begin{equation*}
 \varphi(s)\vcentcolon =\displaystyle{\sum_{j=1}^{n}\alpha_j(s)},\ s\in[0,1].
\end{equation*}
From the conclusion of Remark \ref{rem:CD_parity}, and through 
\eqref{eq:Berry_formula} and \eqref{eq:spherical_coord}, we have that
$\alpha_{j+n}(s)=\alpha_{j}(1-s)$,
for all $j=1,\ldots,n$ and all $s\in[0,1]$. Therefore, using \eqref{eq:sum_of_alphajs}, we have
$\varphi(s)+\varphi(1-s)=2\pi$ for all $s\in[0,1]$, and in
particular $\varphi\left(\dfrac{1}{2}\right)=\pi$.
Finally, note that, taking $s=\hat s\vcentcolon=\dfrac{1}{2}$, we have $z(\hat s,t)=0$, 
and therefore
$$M(x(\hat s,t),y(\hat s,t),z(\hat s,t))=\bmat{0 & A(\gamma(t)) \\ A^*(\gamma(t)) & 0},\ 
\text{ for all } t\in[0,1].$$
From the previous identity, and from the definition of joint-MVD, it follows that, for 
each $j=1,\ldots,n$, the phase $\beta_j$ accrued by the $j$-th singular vectors of the 
joint-MVD of $A(\gamma(\cdot))$ along $\Gamma$ coincides with the Berry phase 
$\alpha_j(\hat s)$ accrued by the eigenvector of $M$ associated to the eigenvalue 
$\lambda_j$ of $M$ along the circle $(x(\hat s,t),y(\hat s,t),0)$. This concludes the 
proof.
\end{proof}

\begin{thm}\label{thm:main_result_nocoal} With the same notation and hypotheses of 
Theorem \ref{thm:main_result} above, except for $A$ being full rank everywhere in 
$\Omega$, we have:
\begin{equation*}
 \sum_{j=1}^n\beta_j=0\mod 2\pi.
\end{equation*}
\end{thm}

\begin{proof} The proof follows the same line as that of Theorem 
\ref{thm:main_result}, using Theorem \ref{thm:generic_coal_nxn}-(i,iii).
\end{proof}
Finally, we have the following result that follows at once from 
Theorems \ref{thm:main_result} and \ref{thm:main_result_nocoal}.
\begin{cor}\label{cor:CircleCoalesce}
Let $A\in\cont^1(\Omega,\Cnxn)$, $\Omega\subset\R^2$, have distinct singular values everywhere on $\Omega$. Let $\Gamma$ be a circle entirely contained in $\Omega$, and $\beta_1,\ldots,\beta_n$ be the phases accrued by the singular vectors of the joint-MVD of $A$ along $\Gamma$. Suppose
\begin{equation*}
 \sum_{j=1}^n\beta_j=\pi\mod 2\pi.
\end{equation*}
Then, there exists a point of loss of rank for $A$ inside the region enclosed by $\Gamma$.
\end{cor}


%

\begin{rem}
Corollary \ref{cor:CircleCoalesce} was formulated relative to a
circle $\Gamma$.  However, this is not necessary.  Using the same homotopy
argument we adopted in \cite{DP_Simax}, it is enough to have $\Gamma$ be
a simple closed curve.
\end{rem}

\section{Examples}\label{sec:examples}

The first example illustrates how Corollary \ref{cor:CircleCoalesce} is used to infer the presence of a point of loss of rank inside the region bounded by a closed curve.

\begin{exm}\label{exm:4x4} 
For this example we have explicitly constructed the $4\times4$ matrix function in \eqref{eq:exmFunction}, where the entries of the matrices $M_0$ to $M_2$ are pseudorandom numbers uniformly distributed in $[-1,1]$, rounded up to the nearest hundredth:
\begin{equation}\label{eq:exmFunction}
A(x,y)=M_0+xM_1+yM_2\,,\ (x,y)\in\R^2,
\end{equation}
with
\begin{equation*}\label{eq:exmMatrices}
\begin{split}
    M_0= &\bmat{
0.03+0.23i & -0.71+0.16i & -0.43-0.53i & 0.90-0.40i\\ 
0.11-0.96i & -0.84-0.16i & 0.26+0.72i & -0.20-0.62i\\ 
0.40-0.98i & 0.96+0.12i & 0.19-0.25i & 0.33+0.26i\\ 
-0.76-0.46i & 0.37-0.22i & -0.50-0.28i & 0.41-0.76i
    }, \\
    M_1= &\bmat{
-0.02-0.79i & 0.48-0.76i & -0.63+0.45i & 0.42-0.03i\\ 
-0.22+0.19i & 0.99+0.05i & -0.80+0.77i & -0.48-0.73i\\ 
0.34+0.83i & 0.94+0.48i & -0.77-0.74i & -0.39-0.26i\\ 
0.10-0.56i & -0.63-0.14i & 0.94+0.00i & 0.09-0.64i
}, \\
    M_2= &\bmat{
0.13+0.40i & 0.43-0.83i & -0.53+0.15i & 0.59+0.59i\\ 
0.07+0.29i & 0.85-0.07i & -0.08-0.86i & 0.75-0.16i\\ 
-0.53+0.26i & 0.17+0.61i & 0.53-0.16i & 0.86-0.83i\\ 
0.95+0.56i & 0.23+0.69i & -0.29-0.99i & -0.46+0.19i
    }.
\end{split}
\end{equation*}
A visual inspection of the surface $\sigma_4(A(x,y))$ suggests the presence of a point of loss of rank for $A$, see Figure \ref{fig:example}. So, we have numerically computed the joint-MVD of $A$ along two circles, a larger one $\Gamma_1$ enclosing the supposed point of loss of rank, and a smaller one $\Gamma_2$ not enclosing the point, see again Figure \ref{fig:example}. Thus, we have computed $\beta_1,\ldots,\beta_4$, i.e. the phases accrued by 
the four columns of 
the unitary factors of the joint-MVD of $A$ along the two circles. Table \ref{table:exm} shows the computed phases, rounded up to the fourth decimal place. The outcome of the computation clearly confirms the expectation of Theorems \ref{thm:main_result} and \ref{thm:main_result_nocoal}: there is a loss of rank inside $\Gamma_1$, but not inside $\Gamma_2$.  
All the computations have been performed using the {\tt MATLAB} function \texttt{complexSvdCont} available at \emph{https://www.mathworks.com/matlabcentral/ fileexchange/160876-smooth-singular-value-decomp-of-complex-matrix-function}. The {\tt MATLAB} code follows closely the algorithm proposed in \cite{DPP2013} for the computation of the MVD of Hermitian matrix functions. In a nutshell, it performs a variable-stepsize continuation of the smooth SVD of a matrix function of one parameter where, at each step, a suitable Procrustes problem is solved to minimize the quantity in \eqref{eq:joint_MVD_integral}.
\begin{table}
    \centering
    \begin{tabular}{|c|c|c|}
        \hline
        & $\Gamma_1$ & $\Gamma_2$\\ \hline
        $\beta_1$ & $-0.0206$ & $+0.7928$\\
        $\beta_2$ & $-2.5572$ & $-0.7905$\\
        $\beta_3$ & $+2.6831$ & $+0.0004$\\
        $\beta_4$ & $+3.0363$ & $-0.0027$\\ \hline
        $\sum\limits_{j=1}^4\beta_j$ & $+3.1416$ & $+0.0000$\\
        \hline
\end{tabular}\caption{Numerically computed phases for Example \ref{exm:4x4}.}\label{table:exm}
\end{table}

\begin{figure}
    \centering
    \includegraphics[width=0.8\linewidth]{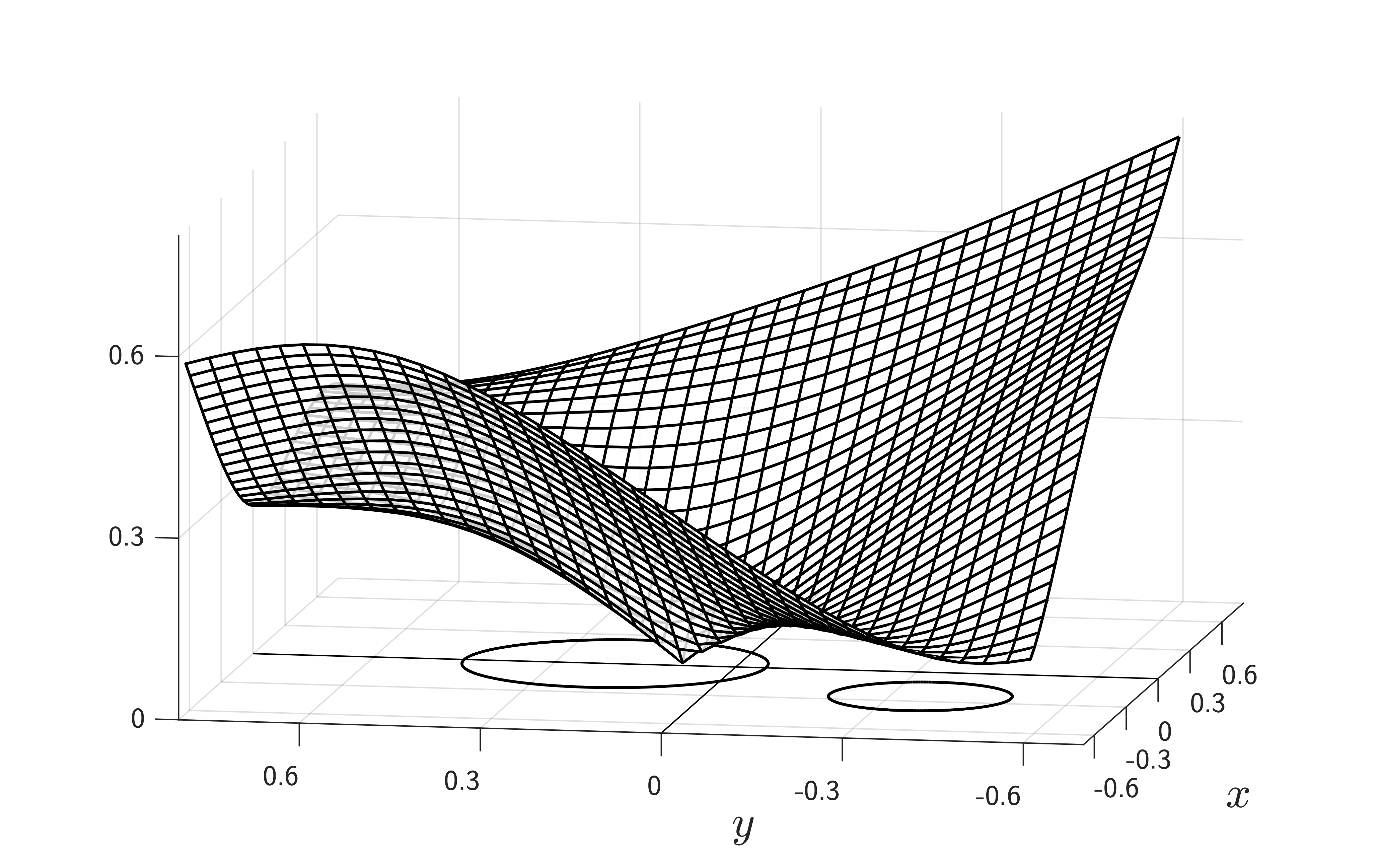}
    \caption{Reference figure for Example \ref{fig:example}: graph of the smallest singular value $\sigma_4(A(x,y))$ and two circles $\Gamma_1$ and $\Gamma_2$ in parameters' space, $\Gamma_1$ being the largest one on the left.}
    \label{fig:example}
\end{figure}
\end{exm}

The next example shows that, in general, by taking the MVD of just
$U$ and/or $V$ will not produce a phase accumulation revealing
the presence of a generic point of loss of rank, and that taking the joint
MVD is necessary.

\begin{exm}\label{exm:NoMVD}
Let $$A(x,y)=\bmat{1 & 1 \\ 0 & x-iy}, (x,y)\in\R^2,$$
and $\Gamma_r$ be the circle parametrized by
$$\gamma(t)=r[\cos(2\pi t),\sin(2\pi t)], r>0, t\in[0,1].$$
Notice that $A$ is full rank everywhere except at the origin $(0,0)$, where it 
has a generic point of loss of rank. By direct computation, it is easy to obtain that:
\begin{itemize}
	\item[i)] letting $\beta_1, \beta_2$ be the phases accrued by, respectively, the first
	 and second column of the unitary factors of the joint-MVD of $A$ 
	 along $\Gamma_r$, one has
	\begin{equation*}
		\beta_1(r) = \pi\frac{r^2}{\left(\frac{1}{2}\left(\sqrt{r^4+4}-r^2\right)+1\right)^2+r^2},
		\quad \beta_2(r) = \pi-\beta_1(r),\quad \text{for all } r\ge 0,
	\end{equation*}
    so that $\sum \beta_j=\pi$ and, in agreement with Corollary \ref{cor:CircleCoalesce}, the point of loss of rank at the origin is properly detected;
 \item[ii)] letting $\alpha_1, \alpha_2$ be the phases accrued by the columns 
	of the unitary factors of the $U$-MVD of $A$ along $\Gamma_r$, one has
	\begin{equation*}
			\alpha_1(r)=2\beta_1(r) \mod 2\pi, \quad \alpha_2(r)=- \alpha_1(r), \quad \text{for all } r\ge 0;
	\end{equation*}
	\item[iii)] no phase is accrued by the columns of the unitary factors of the $V$-MVD 
	of $A$ along $\Gamma_r$, for any value of $r$.
\end{itemize}
In other words, the MVD of just $U$ and/or $V$ does not produce a phase accumulation
revealing the presence of a generic point of loss of rank, whereas the
joint MVD does.  Moreover, to detect the presence of a generic point of loss of rank, 
one has to consider the phase accrued by all singular vectors, as looking solely at
the singular vectors corresponding to the smallest singular value is not sufficient.
\end{exm}

\section{Conclusions}\label{sec:conclusions}
In this work we considered how to detect generic losses of rank for a complex valued matrix function $A$ smoothly depending on two parameters.  We proved that a generic loss of rank is detected by monitoring the (Berry) phases accrued by the singular vectors of an appropriate SVD along closed loops in parameter space containing the value where the loss of rank occurs.  To achieve this, we had to introduce a novel smooth path of the SVD, which we called ``joint MVD'' (joint minimum variation decomposition) for the singular vectors.  We complemented our theoretical results with numerical examples both to locate losses of rank, and to show the necessity of considering the joint MVD.

Although we have considered a single loss of rank within a planar region $\Omega$,
it should be possible to deal with the case of multiple (generic) losses of rank in a 
similar way to how we dealt with multiple coalescing eigenvalues in 
\cite[Section 3]{DP_Simax}, and we plan to look at this problem in the future.

\medskip
%

\end{document}